\documentclass[a4paper]{amsart}
\usepackage{amssymb,latexsym}
\usepackage[utf8x]{inputenc}
\usepackage{amsmath,amsthm}
\usepackage{amsfonts,mathrsfs}
\usepackage[bookmarks=true]{hyperref}

\usepackage[nameinlink,capitalize]{cleveref}
\usepackage{enumerate,units}
\usepackage[all]{xy}

\usepackage{graphicx}
   
\theoremstyle{plain}
\newtheorem{theorem}{Theorem}[section]
\newtheorem{corollary}[theorem]{Corollary}

\newtheorem{fact}[theorem]{Fact}

\newtheorem*{theorem*}{Theorem}

\theoremstyle{definition}
\newtheorem{definition}[theorem]{Definition}

\theoremstyle{remark}
\newtheorem{remark}[theorem]{Remark}

\numberwithin{equation}{section}
\newcommand{\forkindep}[1][]{%
  \mathrel{
    \mathop{
      \vcenter{
        \hbox{\oalign{\noalign{\kern-.3ex}\hfil$\vert$\hfil\cr
              \noalign{\kern-.7ex}
              $\smile$\cr\noalign{\kern-.3ex}}}
      }
    }\displaylimits_{#1}
  }
}

\newcounter{step}                   
    {\hfill $\clubsuit$             
     \vspace{7pt}\par}

\DeclareMathOperator{\ord}{\mathbf{Ord}}
\newcommand{\Vv}{\mathbb{V}}
\newcommand{\Ll}{\mathbb{L}}
\newcommand{\Qq}{\mathbb{Q}}

\DeclareMathOperator{\Th}{Th}
\begin{document}
\title{Saturated Models for the Working Model Theorist}

\begin{abstract}
We put in print a classical result that states that for most purposes, there is no harm in assuming the existence of saturated models in model theory. The presentation is aimed for model theorists with only basic knowledge of axiomatic set theory. 
\end{abstract}
\author{Yatir Halevi}
\address{The Fields Institute for Research in Mathematical Sciences, Toronto, Canada}
\email{ybenarih@fields.utoronto.ca}

\author{Itay Kaplan}
\address{Einstein Institute of Mathematics, Hebrew University of Jerusalem, 91904, Jerusalem
Israel.}
\email{kaplan@math.huji.ac.il}

\thanks{The first author was supported by the Fields Institute for Research in Mathematical Sciences. The second author would like to thank the Israel Science Foundation for their support of this research (grant no. 1254/18).}
\maketitle

\section{Introduction}
Stop a random model theorist on the street and ask him ``are saturated models important/useful?" and his answer will most likely be a strong yes. Nonetheless, the existence of such models (outside the realm of stable theories) will cause some model theorists to move awkwardly in their seats, fearing some set-theoretical obstacles. 

There are several standard workarounds. One is, instead of full saturation, work in some $\kappa$-saturated and $\kappa$-strongly homogeneous model for some large enough cardinal $\kappa$. Another is to consider an increasing chain of highly saturated models, i.e., \emph{special models} and work there instead (see, e.g., \cite[Section 10.4]{hodges}). A third option is to allow your saturated model to be a class instead of a set and work in some suitable conservative expansion of ZFC (see, e.g., \cite[Appendix A]{TZ}).\footnote{One may also consider \emph{resplendent} models and \emph{recursively saturated} models, see \cite[Section 10.6]{hodges}.}

However, we find these workarounds cumbersome, as working with \emph{bona fide} saturated models (as monster models or for checking the completeness of a theory, etc.) is more natural and seamless. There are set theoretic techniques that allow us to do just that.

One such technique is to \emph{force} the continuum hypothesis to hold at some large enough cardinal. Then, after realizing that the statement you wanted to prove was in fact \emph{absolute}, conclude that it holds in your original model of ZFC as well.
Forcing uses some heavy machinery from set theory, and, although this argument is sound, it may deter some model theorists. 

In this short note we give a different argument, whose advantage is that it uses relatively elementary set theory (e.g., no forcing). It is certainly not new nor original: it is standard in axiomatic set theory and is known by many model theorists as well. However, we believe it has never been written down with model theorists in mind.
 
This approach relies on the simple observation that if in the statement one wishes to prove, i.e., a formula $\varphi$ in the language of sets, all quantifiers are bounded, then one can assume an abundance of saturated models. 

More explicitly, let $\Vv$ be a model of ZFC and for an ordinal $\gamma$ let $V_\gamma$ be the $\gamma$-th stage of the von Neumann hierarchy (all notions will be defined in \cref{S:set theory}). 

\begin{theorem*}[\cref{C:how to use}]
For any ordinal $\gamma$ there exists an inner model $M$ satisfying
\begin{enumerate}
\item for some ordinal $\alpha_0$, for every $\alpha\geq \alpha_0$, $M\vDash 2^{\aleph_\alpha}=\aleph_{\alpha+1}$ and
\item for any formula $\varphi$ with parameters in the language of set theory, with parameters and quantifiers bounded by elements of $V_\gamma$, \[\Vv\vDash \varphi\iff M\vDash \varphi.\]
\end{enumerate}
In particular, in $M$ there are arbitrarily large saturated models for any first order theory.
\end{theorem*}

In \cref{S:set theory} we give the necessary background from set theory and prove the existence of such inner models. In \cref{S:model theory} we give some model theoretic examples and applications. 

\subsubsection*{Acknowledgments}
We thank Ehud Hrushovski for making us aware of this method at some point in the many classes he taught us. The second author thanks Yair Hayut and Asaf Karagila for outlining the argument several years ago. We thank Anand Pillay and Christian d\textquoteright Elb\'ee for encouraging us to write it down. We further thank Asaf Karagila and David Meretzky for their comments on a preliminary version.

The work on this note was carried out during the 2021 Thematic Program on Trends in Pure and Applied Model Theory at the Fields Institute. We thank the Fields Institute for their hospitality.

\section{Set Theory}\label{S:set theory}
We assume basic set theory, such as ordinals, cardinals and the axioms of ZFC.  Almost all references will be from Jech's ``Set Theory" \cite{jech}.

We work in some model $(\Vv,\in)$ of ZFC. All sets are elements of $\Vv$ and definable sets are \emph{classes}. Every set is identified with the class of its members, and a class which cannot be identified with a set is called a \emph{proper} class. We let $\ord$ denote the class of all ordinals (recall that $\mathord{\in} \restriction \ord^2$ is also denoted by $<$, it is a well-order); it is a proper class.

The universe $\Vv$ can be partitioned by the so-called \emph{von Neumann hierarchy of sets}: $V_0=\emptyset$, for every $\beta\in \ord$ if $\beta$ is a limit ordinal $V_\beta=\bigcup_{\alpha<\beta} V_\alpha$ and if $\beta=\alpha+1$ then $V_\beta=P(V_\alpha)$. Each $V_\alpha$ is a set, but $\Vv$ itself is a proper class. In fact, it follows from the axioms that $\Vv=\bigcup_{\alpha\in\ord} V_\alpha$.

An \emph{inner model} of ZFC is a class $M$ such that $(M,\mathord{\in})\vDash \text{ZFC}$, $M$ is transitive ($x \in y \in M$ implies $x\in M$), and $M$ contains all ordinals. 


An important family of inner models are the \emph{universes of sets constructible relative to a set}: for any set $A$, we let
\[L_0[A]=\emptyset,\, L_\beta=\bigcup_{\alpha<\beta} L_\beta[A] \text{ if $\beta$ is a limit ordinal and}\]
\[L_{\beta+1}[A]=\{X\subseteq L_\beta[A]: X \text{ is definable in $(L_\beta[A],\mathord{\in}, A\cap L_\beta[A])$}\},\]
\[\Ll [A]=\bigcup_{\beta\in \ord}L_\beta[A],\]
where $(L_\beta[A],\in, A\cap L_\beta[A])$ is the structure $(L_\beta[A],\in)$ with a predicate for $A\cap L_\beta[A]$.
Each $L_\beta[A]$ is a set, but $\Ll[A]$ itself is a proper class. It is an inner model of ZFC and it satisfies the continuum hypothesis for large enough cardinals, i.e., there exists $\alpha_0$ such that for all $\alpha\geq \alpha_0$, $\Ll [A]\vDash 2^{\aleph_\alpha}=\aleph_{\alpha +1}$, see \cite[Theorem 13.22]{jech} and \cite[page 103, 2B]{keith} for more.


The last tool we will need is the following:
\begin{fact}[Mostowski's Collapse]\cite[Theorem 6.15]{jech}
Assume that $P$ is a class and $E$ is a binary relation on $P$ satisfying
\begin{enumerate}
\item $E$ is well-founded:
\begin{enumerate}
\item any nonempty set $x\subseteq P$, has an $E$-minimal element: an element $z\in x$ such that for no $y \in x$ is it the case that $y \mathrel{E} z$;
\item for any $x\in P$, $\{z\in P: z\mathrel{E} x\}$ is a set
\end{enumerate}
\item for any distinct $x,y \in P$, $\{z\in P: z\mathrel{E} x\}\neq \{z\in P: z\mathrel{E} y\}$.
\end{enumerate}
Then there is a transitive class $N$ and an isomorphism $\pi$ between $(P,E)$ and $(N,\in)$. Moreover, $N$ and $\pi$ are unique.
\end{fact}
By the axioms of ZFC, applying Mostowski's collapse to a set produces a transitive set.

\begin{theorem}\label{T:existence of the inner model}
For any ordinal $\gamma$ there exists an inner model $M\vDash$ ZFC satisfying 
\begin{enumerate}
\item $V_\gamma\in M$ and 
\item there exists $\alpha_0$ such that for all $\alpha\geq \alpha_0$, $M\vDash 2^{\aleph_\alpha}=\aleph_{\alpha+1}$.
\end{enumerate}
\end{theorem}
\begin{remark}
By \cite[Theorem 13.27.i]{jech}, $V_\alpha^M=V_\alpha$ for any $\alpha\leq \gamma$, where $V_\alpha^M$ is the $\alpha$-th stage of the hierarchy as defined in $M$.
\end{remark}
\begin{proof}
Let $f:\mu\to V_\gamma$ be a bijection between $V_\gamma$ and some cardinal $\mu$ and let $R=f^{-1}(\mathord{\in}\restriction V_\gamma^2)\subseteq \mu\times \mu$. Note that $R$ is a set (as a subclass of a set).

Let $M=\Ll[R]$ be the constructible universe relative to $R$. It is an inner model of ZFC and satisfies (2), so we need to show that it satisfies (1) as well.

Since $M$ is an inner model, not only $\mu\in M$ but also $\mu\times \mu\in M$. Note that the construction of the product $\mu\times\mu$ is exactly the same in $M$ and in $\Vv$. Assume that $\mu\times\mu \in L_\beta[R]$ for some $\beta\in\ord$. By transitivity $R\subseteq L_{\beta}[R]$, which means, by definition, that $R\in L_{\beta+1}[R]\subseteq M$. 


Consider the structure $(\mu,R)$. It clearly satisfies the requirements of Mostowski's collapse in $\Vv$ since $f(R)=\mathord{\in}$, and hence also in $M$ (by transitivity). By uniqueness, in $\Vv$, the collapse is witnessed by $V_\gamma$ and $f$. 

Applying Mostowski's collapse in $M$, we get a pair $\pi$ and $(N,\in)$. By transitivity, $\pi$ is an isomorphism from $(\mu,R)$ to $(N,\in)$ also in $\Vv$. Uniqueness (in $\Vv$) implies that $\pi=f$ and $N=V_\gamma$, guaranteeing that $V_\gamma\in M$.
\end{proof}

\begin{definition}
For an ordinal $\gamma$, a formula over $V_\gamma$ is \emph{bounded by $V_\gamma$} (or \emph{$V_\gamma$-bounded}) if all its quantifiers are of the form $(\operatorname{Q} x\in a)$ where $\operatorname{Q} \in \{\forall,\exists\}$ and $a\in V_\gamma$.
\end{definition}

\begin{corollary}\label{C:how to use}
For any ordinal $\gamma$ there exists an inner model $M\vDash$ ZFC such that $V_\gamma\in M$ and:
\begin{enumerate}
\item there exists $\alpha_0$ such that for all $\alpha\geq \alpha_0$, $M\vDash 2^{\aleph_\alpha}=\aleph_{\alpha+1}$ and
\item for any $V_\gamma$-bounded formula $\varphi$, $\Vv\vDash\varphi$ if and only if $M\vDash\varphi$.
\end{enumerate}
In particular, in $M$ there are arbitrarily large saturated models for any first order theory.
\end{corollary}
\begin{proof}
Follows from \cref{T:existence of the inner model} since $M$ and $V_\gamma$ are transitive: $\Vv\vDash\varphi$ if and only if $V_\gamma \vDash \varphi$ if and only if $M\vDash\varphi$. (These equivalences follow from an easy inductive argument, this is \emph{absoluteness} for bounded formulas \cite[Lemma 12.9]{jech}.)

The existence of saturated models follows by, e.g., \cite[Lemma 6.1.2]{TZ}.
\end{proof}

\section{Model Theory}\label{S:model theory}
The aim of this section is to give some examples of applications of \cref{C:how to use} to ``real-life'' model-theoretic statements. One can summarize it in one slogan:
\begin{quotation}
\emph{If the statement you wish to prove is equivalent to a $V_\gamma$-bounded formula (with parameters) in any inner model containing $V_\gamma$ then you may assume that you have arbitrarily large saturated models.}
\end{quotation}

\begin{remark}
\begin{enumerate}
\item The point that the equivalence is with respect to any inner model containing $V_\gamma$ is crucial because of the reflection principle (\cite[Theorem 12.14]{jech}) which says that for every formula with parameters $\varphi$ there is an ordinal $\gamma$ and a $V_\gamma$-bounded formula $\psi$ such that $\Vv\vDash \varphi\leftrightarrow \psi$ (see \cite[Theorem 12.14]{jech}). 
\item Assuming large cardinals, there are models of ZFC in which e.g., $\Th(\Qq,<)$ has no uncountable saturated models, see the non-example below.
\end{enumerate}
\end{remark}

Many theorems in model theory fall into this category, and hence we may assume the existence of saturated models, and do not have to worry about the existence of monster models. We give two examples and one non-example.

\subsubsection*{{Completeness}} \label{sss:complete}
One way to show that a theory $T$ in a language $\mathcal{L}$ is complete is by showing that any two saturated models of the same cardinality are isomorphic (assuming one exists). 

We can easily write down a formula in the language of set theory saying that the theory $T$ is complete. Pick an ordinal $\gamma$ large enough such that $V_\gamma$ contains, among other things, $T$, $\mathcal{L}$, the set of all $\mathcal{L}$-formulas and the set of all finite sequences of formulas (and in particular all deductions). Clearly the formula ``$T$ is complete'' is equivalent to a $V_{\gamma}$-bounded formula (and this remains true in any inner model containing $V_\gamma$). By \cref{C:how to use}, in order to prove that $T$ is complete we may assume that we have a saturated model for $T$.

Another approach to writing ``$T$ is complete" in a formula is: for any model $M\vDash T$ whose universe is the cardinal $|T|$ if $M\vDash \phi$ then $\phi\in T$. This is also easily seen to be a $V_\gamma$-bounded for some $\gamma$.

As a consequence, when showing quantifier elimination (i.e., substructure completeness) and model completeness one may assume the existence of saturated models.

\subsubsection*{{Stably Embedded Definable Sets}}
Let $T$ be a theory. Recall the definition of a stably embedded definable set (see, e.g., \cite[Appendix]{stabemb}): a formula $\xi(x)$ is stably embedded if for any formula $\varphi(x,y)$ there exists a formula $\psi(x,y)$ such that for any $y$-tuple $b$ there exists a $z$-tuple $d$ satisfying $\xi$ such $\forall x (\xi(x)\to (\varphi(x,b)\leftrightarrow \psi(x,d)))$. Again, we can easily express this as a formula in the language of set theory with $V_\gamma$-bounded quantifiers for some sufficiently large $\gamma$.

If $M\vDash T$ is saturated then by \cite[Lemma 1, Appendix]{stabemb}, a definable set $\xi(M)$ is stably embedded if and only if every automorphism of $\xi(M)$ (with the induced structure) extends to an automorphism of $M$. This is often a very convenient criterion to check.

\subsubsection*{{Non-Example}}
Assuming large cardinals, there is a model of ZFC and a theory $T$ for which there is no ordinal $\gamma$ such that the statement ``$T$ has an uncountable saturated model'' is equivalent to a $V_\gamma$-bounded formula in all inner models. 

Indeed, by \cite{genfails} there are models of ZFC in which the generalized continuum hypothesis fails for every cardinal (assuming a supercompact cardinal). By restricting to $V_\kappa$ where $\kappa$ is the smallest inaccessible cardinal, we may assume that there are no inaccessible cardinals as well (see \cite[Lemma 12.13]{jech}). In such a model, any non-stable theory $T$ will not have any uncountable saturated models: if $T$ had an uncountable saturated model of size $\lambda$, then by \cite[Theorem VIII.4.7]{classification}, $\lambda=\lambda^{<\lambda}$, and hence $\lambda$ is regular. As $\lambda$ is uncountable and not inaccessible, it is a successor $\mu^+$. Together, we get $2^\mu \leq \lambda^\mu = \lambda = \mu^+ \leq 2^\mu$, a contradiction. Consequently, if ``$T$ has an uncountable saturated model'' was equivalent to a $V_\gamma$-bounded formula (in all inner models) it would contradict \cref{C:how to use}.
%
\bibliographystyle{alpha}
\bibliography{saturated}

\end{document}